\documentclass[12pt]{amsart}
\usepackage{latexsym}
\usepackage{amsfonts}
\usepackage{amssymb}
\usepackage{mathrsfs}

\newtheorem{theorem}{Theorem}[section]

\newtheorem{lemma}[theorem]{Lemma}

\newtheorem{fact}[theorem]{Fact}
\theoremstyle{remark}

\numberwithin{equation}{section}

\newcommand{\ra}{\rightarrow}
\newcommand{\mt}{\mapsto}
\newcommand{\ZZ}{\mathbb{Z}}

\newcommand{\Ad}{\mathrm{Ad}}
\newcommand{\Lp}{\mathfrak{p}}
\newcommand{\Lq}{\mathfrak{q}}
\newcommand{\Ls}{\mathfrak{s}}

\newcommand{\LL}{\mathscr{L}}

\begin{document}

\title{Nonabelian cohomology of compact Lie
groups}

\author{Jinpeng An}
\address{School of mathematical sciences, Peking University,
 Beijing, 100871, China }
\email{anjinpeng@gmail.com}

\author{Ming Liu}
\address{School of mathematical sciences, Peking University,
 Beijing, 100871, China}
\email{mingliulm@yahoo.com.cn}

\author{Zhengdong Wang}
\address{School of mathematical sciences, Peking University,
 Beijing, 100871, China}
\email{zdwang@pku.edu.cn}

\keywords{nonabelian cohomology, compact Lie group, maximal compact
subgroup.}

\subjclass[2000]{20J06; 22E15; 57S15.}

\begin{abstract}
Given a Lie group $G$ with finitely many components and a compact
Lie group $A$ which acts on $G$ by automorphisms, we prove that
there always exists an $A$-invariant maximal compact subgroup $K$ of
$G$, and that for every such $K$, the natural map $H^1(A,K)\ra
H^1(A,G)$ is bijective. This generalizes a classical result of Serre
\cite{Se} and a recent result in \cite{AW}.
\end{abstract}

\maketitle

\section{Introduction}

Let a group $A$ act on a Lie group $G$ by automorphisms. Recall that
the set of cocycles $Z^1(A,G)$ consists of maps $\gamma:A\ra G$
satisfying $\gamma(ab)=\gamma(a)a(\gamma(b))$ for all $a,b\in A$,
and that $\gamma_1,\gamma_2\in Z^1(A,G)$ are cohomologous if for
some $g\in G$ we have $\gamma_2(a)=g^{-1}\gamma_1(a)a(g)$ for all
$a\in A$. The first nonabelian cohomology $H^1(A, G)$ of $A$ with
coefficients in $G$ is, by definition, the set of all cohomologous
classes in $Z^1(A,G)$ (c.f. \cite{Se}).

Because of its relation with number theory, most studies of this
kind of cohomology concentrate on the case that $G$ is also
algebraic. For example, a classical result of Serre
\cite[III.4.5]{Se} asserts that if $G$ is a complex reductive
algebraic group with a maximal compact subgroup $K$, and
$A\cong\ZZ/2\ZZ$ acts on $G$ by the complex conjugation with respect
to $K$, then the natural map $H^1(A,K)\ra H^1(A,G)$ is bijective.
Recently, the case that $G$ is an arbitrary connected Lie group was
considered in \cite{AW,An}. In particular, it was proved that for
any finite group $A$ and any connected Lie group $G$, there exists
an $A$-invariant maximal compact subgroup $K$ of $G$, and the
natural map $H^1(A,K)\ra H^1(A,G)$ is bijective (\cite[Thm.
3.1]{AW}).

The goal of this paper is to generalize the above results to the
case that $A$ is an arbitrary compact Lie group and $G$ has finitely
many components. In this setting, by a cocycle $\gamma:A\ra G$ we
always mean a continuous one. Our main theorem is as follows.

\begin{theorem}\label{T:main}
Let $G$ be a Lie group with finitely many components, and let $A$ be
a compact Lie group which acts on $G$ by automorphisms. Then there
exists an $A$-invariant maximal compact subgroup $K$ of $G$, and for
every such $K$, then natural map $\iota_1:H^1(A,K)\ra H^1(A,G)$ is
bijective.
\end{theorem}

It should be pointed out that the main difficulty in the proof of
Theorem \ref{T:main} lies in the injectivity of the map $\iota_1$.
Recall that the proof of the corresponding part for the special case
of Theorem \ref{T:main} where $A$ is finite (and $G$ is connected),
which is \cite[Thm. 3.1]{AW}, is based on the well-known fact that
if $K$ is a maximal compact subgroup of a Lie group $G$ with
finitely many components, then there exist $\Ad_G(K)$-invariant
subspaces $\Lp_1,\ldots,\Lp_r$ of $\LL(G)$ with
$\LL(G)=\LL(K)\oplus\Lp_1\oplus\cdots\oplus\Lp_r$ such that the map
$K\times\Lp_1\times\cdots\times\Lp_r\rightarrow G$,
$(k,X_1,\ldots,X_r)\mt ke^{X_1}\cdots e^{X_r}$ is a diffeomorphism
(c.f. \cite[Thm. XV.3.1]{Ho}). (Throughout this paper, $\LL$ denotes
the functor which takes a Lie group to its Lie algebra.) To prove
the injectivity of $\iota_1$ in Theorem \ref{T:main}, we need the
following generalization of this fact: If a compact Lie group $A$
acts on $G$ by automorphisms and $K$ is $A$-invariant, then the
subspaces $\Lp_1,\ldots,\Lp_r$ can be chosen to be $A$-invariant
(for the precise statement, c.f. Lemma \ref{L:Hochschild} below). We
will prove this result in Section 2. Theorem \ref{T:main} will be
proved in Section 3.

\section{Some lemmas}

In this section we prove Lemmas \ref{L:invariant} and
\ref{L:Hochschild} below. Lemma \ref{L:invariant} ensures the first
assertion in Theorem \ref{T:main}, and Lemma \ref{L:Hochschild} is a
crucial tool to prove the injectivity of the map $\iota_1$ in
Theorem \ref{T:main}. We need the following well-known fact (c.f.
\cite[Thm. XV.3.1]{Ho} or \cite[Thm. VII.1.2]{Bo}).

\begin{fact}\label{F:conjugate}
Let $G$ be a Lie group with finitely many components, and let $K$ be
a maximal compact subgroup of $G$. Then any compact subgroup of $G$
can be conjugated into $K$ by $G$.
\end{fact}

\begin{lemma}\label{L:invariant}
Let $G$ and $A$ be as in Theorem \ref{T:main}, and let $L$ be an
$A$-invariant compact subgroup of $G$. Then there exists an
$A$-invariant maximal compact subgroup $K$ of $G$ which contains
$L$.
\end{lemma}

\begin{proof}
Denote $H=G\rtimes A$, and view $G$ and $A$ as subgroups of $H$ in
the natural way. Since $L\rtimes A$ is a compact subgroup of $H$,
there exists a maximal compact subgroup $M$ of $H$ which contains
$L\rtimes A$. Let $K=M\cap G$. It is obvious that the compact group
$K$ contains $L$ and is $A$-invariant. We claim that $K$ is a
maximal compact subgroup of $G$. Indeed, for any compact subgroup
$K'$ of $G$, by Fact \ref{F:conjugate}, there exists $h\in H$ with
$hK'h^{-1}\subset M$. But since $G$ is normal in $H$, we also have
$hK'h^{-1}\subset G$. Thus $hK'h^{-1}\subset M\cap G=K$. This proves
that $K$ is maximal compact in $G$.
\end{proof}

Our proof of the following lemma is motivated by that of \cite[Thm.
XV.3.1]{Ho}.

\begin{lemma}\label{L:Hochschild}
Let $G$ and $A$ be as in Theorem \ref{T:main}, and let $K$ be an
$A$-invariant maximal compact subgroup of $G$. Then there exist
linear subspaces $\Lp_1,\ldots,\Lp_r$ of $\LL(G)$ which are
invariant under both $\Ad_G(K)$ and $A$ such that
$\LL(G)=\LL(K)\oplus\Lp_1\oplus\cdots\oplus\Lp_r$, and such that the
map $\varphi:K\times\Lp_1\times\cdots\times\Lp_r\rightarrow G$
defined by $$\varphi(k,X_1,\ldots,X_r)=ke^{X_1}\cdots e^{X_r}$$ is a
diffeomorphism.
\end{lemma}

\begin{proof}
We may assume that $G$ is noncompact, and prove the lemma by
induction on $\dim G$. Firstly, if $G$ is not semisimple, we define
a nontrivial $A$-invariant closed connected normal abelian subgroup
$C$ of $G$ as follows. Let $\Ls$ be the solvable radical of
$\LL(G)$, and define $\Ls_i$ inductively as $\Ls_0=\Ls$ and
$\Ls_i=[\Ls_{i-1},\Ls_{i-1}]$. Since $\Ls$ is solvable, there exists
$d$ such that $\Ls_d\ne0$ and $\Ls_{d+1}=0$. Then $\Ls_d$ is
abelian. Let $S_d$ be the connected Lie subgroup of $G$ with Lie
algebra $\Ls_d$. Then we define $C$ as the closure of $S_d$ in $G$.
It is obvious that $C$ satisfies the required properties.

Now we define an $A$-invariant closed normal abelian subgroup $D$ of
$G$ according to the following three cases.
\begin{itemize}
\item[(1)] If $G$ is semisimple, we define $D=Z(G_0)$.
\item[(2)] If $G$ is not semisimple and $C$ is a vector group, we define $D=C$.
\item[(3)] If $G$ is not semisimple and $C$ is not a vector group, we define $D$ as the unique maximal compact subgroup of $C$.
\end{itemize}
Let $G'=G/D$, and let $\pi:G\ra G'$ be the quotient homomorphism.
Then $A$ acts on $G'$ by automorphisms. By Lemma \ref{L:invariant},
we can choose an $A$-invariant maximal compact subgroup $K'$ of $G'$
which contains $\pi(K)$. Let $H=\pi^{-1}(K')$, which is an
$A$-invariant subgroup of $G$. Clearly, $K\subset H$ is a maximal
compact subgroup of $H$.

We first prove the following two claims.

\emph{Claim 1.} Lemma \ref{L:Hochschild} holds for the pair
$(G',K')$.

For case (1), $G'_0$ is semisimple with trivial center. We choose
$\Lp'$ as the orthogonal complement of $\LL(K')$ in $\LL(G')$ with
respect to the Killing form of $\LL(G')$. Since the Killing form is
preserved by any automorphism, we see that $\Lp'$ is invariant under
both $\Ad_{G'}(K')$ and $A$. It is well-known that
$\LL(G')=\LL(K')\oplus\Lp'$, and the Cartan decomposition ensures
that the map $\varphi':K'\times\Lp'\ra G'$,
$\varphi'(k',X')=k'e^{X'}$ is a diffeomorphism (for the case that
$G'$ is non-connected, c.f. \cite[Prop. VII.2.3]{Bo}). For the last
two cases, we have $\dim G'<\dim G$, and Claim 1 follows from the
induction hypothesis. This finishes the verification of Claim 1.

\emph{Claim 2.} Lemma \ref{L:Hochschild} holds for the pair $(H,K)$.

For case (1), it is well-known that $\pi^{-1}(K'_0)$ is connected
(c.f. \cite[Thm. 6.31(e)]{Kn}). So $H$ has finitely many components.
Since $G$ is noncompact, we have $\dim K'<\dim G'$. So $\dim H<\dim
G$, and in this case Claim 2 follows from the induction hypothesis.
For case (2), since $D$ is a closed normal vector subgroup of $H$
and $H/D$ is compact, by a theorem of Iwasawa (c.f. \cite[Thm.
III.2.3 and Lem. XV.3.2]{Ho} or \cite[Thm VII.4.1]{Bo}), we indeed
have $H=D\rtimes K$. If we let $\Lp=\LL(D)$, then the map
$\varphi_0:K\times\Lp\ra H$, $\varphi_0(k,X)=ke^X$ is obviously a
diffeomorphism. This proves Claim 2 for case (2). For case (3),
since $D$ is compact, $H$ is also compact, and there is nothing to
prove.

Now we have a surjective $A$-equivariant homomorphism $\pi:G\ra G'$
with kernel $D$, an $A$-invariant maximal compact subgroup $K'$ of
$G'$ containing $\pi(K)$, subspaces $\Lp'_1,\ldots,\Lp'_m$ of
$\LL(G')$, and subspaces $\Lp_1,\ldots,\Lp_n$ of $\LL(H)$, where
$H=\pi^{-1}(K')$, such that
\begin{itemize}
\item[(1)] $\LL(G')=\LL(K')\oplus\Lp'_1\oplus\cdots\oplus\Lp'_m$ and
$\LL(H)=\LL(K)\oplus\Lp_1\oplus\cdots\oplus\Lp_n$.
\item[(2)] Every $\Lp'_i$ is invariant
under $\Ad_{G'}(K')$ and $A$, and every $\Lp_j$ is invariant under
$\Ad_G(K)$ and $A$.
\item[(3)] The maps
$\varphi':K'\times\Lp'_1\times\cdots\times\Lp'_m\rightarrow G'$ and
$\varphi_0:K\times\Lp_1\times\cdots\times\Lp_n\rightarrow H$ defined
by
\begin{align*}
\varphi'(k',X'_1,\ldots,X'_m)&=k'e^{X'_1}\cdots e^{X'_m},\\
\varphi_0(k,X_1,\ldots,X_n)&=ke^{X_1}\cdots e^{X_n}
\end{align*}
are diffeomorphisms.
\end{itemize}
Note that the compact group $K\rtimes A$ acts linearly on each
$(d\pi)^{-1}(\Lp'_i)$, and the subspace $\LL(D)$ of
$(d\pi)^{-1}(\Lp'_i)$ is invariant under $K\rtimes A$. So there
exists a subspace $\Lq_i$ of $(d\pi)^{-1}(\Lp'_i)$ which is
invariant under $K\rtimes A$ such that
$(d\pi)^{-1}(\Lp'_i)=\LL(D)\oplus\Lq_i$. Now the subspaces $\Lp_j$
and $\Lq_i$ are all invariant under $\Ad_G(K)$ and $A$, and it is
easy to see that
$\LL(G)=\LL(K)\oplus\Lp_1\oplus\cdots\oplus\Lp_n\oplus\Lq_1\oplus\cdots\oplus\Lq_m$.
It remains to prove that the map
$\varphi:K\times\Lp_1\times\cdots\times\Lp_n\times\Lq_1\times\cdots\times\Lq_m\rightarrow
G$ defined by
$$\varphi(k,X_1,\ldots,X_n,Y_1,\ldots,Y_m)=ke^{X_1}\cdots e^{X_n}e^{Y_1}\cdots e^{Y_m}$$ is a
diffeomorphism. Since $\varphi_0$ and $\varphi'$ are
diffeomorphisms, there are smooth maps $\overline{k}:H\ra K$,
$\overline{X_j}:H\ra\Lp_j$, $\overline{k'}:G'\ra K'$, and
$\overline{X'_i}:G'\ra\Lp'_i$ such that
\begin{align*}
h&=\overline{k}(h)e^{\overline{X_1}(h)}\cdots e^{\overline{X_n}(h)},\\
g'&=\overline{k'}(g')e^{\overline{X'_1}(g')}\cdots
e^{\overline{X'_m}(g')}
\end{align*}
for all $h\in H$ and $g'\in G'$. We define smooth maps
$\widetilde{Y_i}:G\ra\Lq_i$, $\widetilde{h}:G\ra H$,
$\widetilde{k}:G\ra K$, $\widetilde{X_j}:G\ra\Lp_j$ as
\begin{align*}
\widetilde{Y_i}&=(d\pi|_{\Lq_i})^{-1}\circ\overline{X'_i}\circ\pi,\\
\widetilde{h}(g)&=g(e^{\widetilde{Y_1}(g)}\cdots e^{\widetilde{Y_m}(g)})^{-1}\in H,\\
\widetilde{k}&=\overline{k}\circ\widetilde{h},\\
\widetilde{X_j}&=\overline{X_j}\circ\widetilde{h}.
\end{align*}
Let
$$\psi=(\widetilde{k},\widetilde{X_1},\ldots,\widetilde{X_n},\widetilde{Y_1},\ldots,\widetilde{Y_m}):G\ra
K\times\Lp_1\times\cdots\times\Lp_n\times\Lq_1\times\cdots\times\Lq_m.$$
Then it is straightforward to check that both $\varphi\circ\psi$ and
$\psi\circ\varphi$ are the identity maps. Thus $\varphi$ is a
diffeomorphism. The proof of Lemma \ref{L:Hochschild} is finished.
\end{proof}

\section{Proof of the main theorem}

Now we prove our main Theorem \ref{T:main}.

\begin{proof}[Proof of Theorem \ref{T:main}]
The first assertion has been proved in Lemma \ref{L:invariant}. Now
we prove the surjectivity of $\iota_1:H^1(A,K)\ra H^1(A,G)$. We
first recall that the group operations in $G\rtimes A$ are defined
as
$$(g,a)(h,b)=(ga(h),ab), \quad (g,a)^{-1}=(a^{-1}(g^{-1}),a^{-1}).$$
We claim that for every $A$-invariant maximal compact subgroup $K$
of $G$, $K\rtimes A$ is a maximal compact subgroup of $G\rtimes A$.
Indeed, if $L$ is a compact subgroup of $G\rtimes A$ containing
$K\rtimes A$, then $L\cap G$ is a compact subgroup of $G$ containing
$K$. This forces $L\cap G=K$. Now if $h=(g,a)\in L$, since $A\subset
L$, we have $g=ha^{-1}\in L\cap G=K$. Hence $h\in K\rtimes A$. This
proves that $K\rtimes A$ is maximal compact in $G\rtimes A$. Now let
$\gamma:A\ra G$ be a cocycle. Then it is easy to check that the map
$\widetilde{\gamma}:A\ra G\rtimes A$ defined as
$\widetilde{\gamma}(a)=(\gamma(a),a)$ is a homomorphism. Since
$\widetilde{\gamma}$ is continuous, we see that
$\widetilde{\gamma}(A)$ is a compact subgroup of $G\rtimes A$. Hence
there exists $(g,b)\in G\rtimes A$ such that
$(g,b)^{-1}\widetilde{\gamma}(A)(g,b)\subset K\rtimes A$. This means
that $(g,b)^{-1}(\gamma(a),a)(g,b)\in K\rtimes A$ for all $a\in A$.
But
$$(g,b)^{-1}(\gamma(a),a)(g,b)=(b^{-1}(g^{-1}\gamma(a)a(g)),b^{-1}ab).$$
So we have $g^{-1}\gamma(a)a(g)\in K$ for all $a\in A$. Hence
$\gamma$ is cohomologous to a cocycle which takes values in $K$.
This proves that $H^1(A,K)\ra H^1(A,G)$ is surjective.

To prove that $\iota_1$ is injective, let $\gamma_1,\gamma_2:A\ra K$
be cocycles which are cohomologous under $G$, i.e., there exists
$g\in G$ with $\gamma_2(a)=g^{-1}\gamma_1(a)a(g)$ for all $a\in A$.
By Lemma \ref{L:Hochschild}, there exist linear subspaces
$\Lp_1,\ldots,\Lp_r$ of $\LL(G)$ which are invariant under
$\Ad_G(K)$ and $A$ such that
$\LL(G)=\LL(K)\oplus\Lp_1\oplus\cdots\oplus\Lp_r$, and such that the
map $\varphi:K\times\Lp_1\times\cdots\times\Lp_r\rightarrow G$
defined by $$\varphi(k,X_1,\ldots,X_r)=ke^{X_1}\cdots e^{X_r}$$ is a
diffeomorphism. Write $g$ as $g=\varphi(k,X_1,\ldots,X_r)$. Then for
any $a\in A$, we compute
\begin{align*}
&\varphi(\gamma_2(a)^{-1}k\gamma_2(a),\Ad(\gamma_2(a)^{-1})(X_1),\ldots,\Ad(\gamma_2(a)^{-1})(X_1))\\
=&\gamma_2(a)^{-1}g\gamma_2(a)\\
=&\gamma_2(a)^{-1}\gamma_1(a)a(g)\\
=&\gamma_2(a)^{-1}\gamma_1(a)a(k)e^{da(X_1)}\cdots e^{da(X_r)}\\
=&\varphi(\gamma_2(a)^{-1}\gamma_1(a)a(k),da(X_1),\ldots,da(X_r)).
\end{align*}
Since $\varphi$ is injective, we get
$$\gamma_2(a)^{-1}k\gamma_2(a)=\gamma_2(a)^{-1}\gamma_1(a)a(k).$$
This means that $\gamma_2(a)=k^{-1}\gamma_1(a)a(k)$. So $\gamma_1$
and $\gamma_2$ are cohomologous under $K$. This proves the
injectivity of $\iota_1$.
\end{proof}

\textbf{Acknowledgment.} The authors were partially supported by
NSFC Grant 10771003.


\begin{thebibliography}{99}

\bibitem{AW} J. An, Z. Wang, \textit{Nonabelian cohomology with coefficients in
Lie groups}, Trans. Amer. Math. Soc., 360 (2008), no. 6, 3019--3040.

\bibitem{An}  J. An, \textit{Twisted Weyl groups of Lie groups and nonabelian cohomology}, Geom. Dedicata, 128 (2007), 167--176.

\bibitem{Bo}  A. Borel, \textit{Semisimple groups and Riemannian symmetric spaces}, Hindustan
Book Agency, New Delhi, 1998.

\bibitem{Ho}  G. Hochschild, \textit{The structure of Lie groups},
Holden-Day, San Francisco, 1965.

\bibitem{Kn}  A. W. Knapp, \textit{Lie groups beyond an introduction},
second edition, Birkh\"{a}user, Boston, 2002.

\bibitem{Se}  J.-P. Serre, \textit{Galois cohomology}, Springer-Verlag, Berlin,
1997.

\end{thebibliography}
\end{document}